\documentclass[12 pt]{amsart}
\usepackage{amsmath,amssymb,amsthm,hyperref}
\usepackage{longtable,}
\usepackage{a4wide}
\usepackage{color}
\usepackage{enumerate}
\usepackage{cleveref}

\usepackage{tikz}
\usetikzlibrary{decorations.pathreplacing}

\newtheorem{thm}{Theorem}[section]
\newtheorem*{thm*}{Theorem}
\newtheorem*{conj*}{Conjecture}
\newtheorem{cor}[thm]{Corollary}

\newtheorem{lem}[thm]{Lemma}
\newtheorem{prop}[thm]{Proposition}

\theoremstyle{remark}

\theoremstyle{definition}

\newtheorem{defn}[thm]{Definition}
\newcounter{claim}[thm]

\DeclareMathOperator{\Aut}{Aut}
\DeclareMathOperator{\Cay}{Cay}

\newcommand{\ZZ}{\mathbb{Z}}

\crefname{lem}{lemma}{lemmas}
\Crefname{lem}{Lemma}{Lemmas}
\crefmultiformat{lem}{Lemmas~#2#1#3}{ and~#2#1#3}{, #2#1#3}{, and~#2#1#3}

\title{Most Rigid Representations and Cayley Index}
\author{Joy Morris}
\author{Josh Tymburski}

\address{Department of Mathematics and Computer Science\\
University of Lethbridge\\
Lethbridge, AB T1K 3M4\\
Canada}

\email{joy.morris@uleth.ca}\email{josh.tymburski@uleth.ca}	
\begin{document}

\begin{abstract}
For any finite group $G$, a natural question to ask is the order of the smallest possible automorphism group for a Cayley graph on $G$. A particular Cayley graph whose automorphism group has this order is referred to as an MRR (Most Rigid Representation), and its Cayley index is a numerical indicator of this value. Study of GRRs showed that with the exception of two infinite families and seven individual groups, every group admits a Cayley graph whose MRR is a GRR, so that the Cayley index is $1$. The full answer to the question of finding the smallest possible Cayley index for a Cayley graph on a fixed group was almost completed in previous work, but the precise answers for some finite groups and one infinite family of groups were left open. We fill in the remaining gaps to completely answer this question.
\end{abstract}

\maketitle

\section{Introduction}
All groups and graphs in this paper are finite. All of our graphs are simple, undirected, and have no loops.

A Cayley graph $\Gamma=\Cay(G,S)$ where $S \subseteq G$ with $S=S^{-1}$ and $1 \notin S$, is the graph whose vertices are the elements of $G$, with $(g,gs) \in E(\Gamma)$ if and only if $g \in G$ and $s \in S$. We refer to $S$ as the connection set for $\Gamma$. Let $A=\Aut(\Gamma)$. Observe that $L_G$, the left-regular representation of $G$, lies in $A$, so $|G|$ divides $|A|$. 

\begin{defn}
The \emph{Cayley index} $c(\Gamma)$ of the Cayley graph $\Gamma=\Cay(G,S)$, is $|A:L_G|$. The \emph{Cayley index} $c(G)$ of the group $G$ is $\min_{S \subseteq G, S=S^{-1}} c(\Cay(G,S))$; that is, the lowest Cayley index of any Cayley graph on the group $G$.
\end{defn} 

\begin{defn}
A Cayley graph $\Gamma=\Cay(G,S)$ is a GRR (Graphical Regular Representation) for $G$ if $c(\Gamma)=1$.
\end{defn}

Thus, groups that admit GRRs are precisely the groups whose Cayley index is $1$. In order to completely characterise these groups, we require another definition.

\begin{defn}
Let $A$ be an abelian group of even order, and $y$ an involution in $A$. Then the generalised dicyclic group $Dic(A,y,x)$ is $\langle A,x\rangle$ where $x \notin A$, $x^2=y$, and $x^{-1}ax=a^{-1}$ for every $a \in A$.
\end{defn}

The study of GRRs involved many researchers and papers. Some of the most influential work along the way appeared in \cite{Imrich1, Imrich2, Imrich3}. Watkins \cite{Watkins} observed that there are two infinite families of graphs that cannot admit GRRs: generalised dicyclic groups, and abelian groups that are not elementary abelian $2$-groups. Determining groups that admit GRRs was completed by Hetzel \cite{Hetzel} and Godsil \cite{Godsil}. Hetzel proved that aside from the two infinite families noted by Watkins, there are exactly 13 small solvable groups (of order at most 32) that do not admit a GRR. Godsil showed that every non-solvable group admits a GRR.

In the case where a group fails to admit a GRR, a natural question to ask is: what is the Cayley index of the group, and what is a Cayley graph on the group that has that Cayley index? The following terminology was coined in \cite{ImrichW}.

\begin{defn}
Let $G$ be a group with $c(G)>1$, and let $\Gamma=\Cay(G,S)$ be a Cayley graph on $G$ with $c(\Gamma)=c(G)$. Then we say that $\Gamma$ is an MRR (Most Rigid Representation) for $G$.
\end{defn}

The bulk of this paper is divided into 4 sections. In~\Cref{sec:exceptional}, we describe the groups that do not admit a GRR but do not lie in either of the infinite families of groups that do not admit a GRR. For each of these groups, we find its Cayley index and an MRR. In~\Cref{sec:abelian}, we find the Cayley index of every abelian group, and find MRRs for those groups whose Cayley index is greater than $2$. In~\Cref{sec:Q}, we consider a subfamily of generalised dicyclic groups (specifically, the hamiltonian $2$-groups), and show that the smallest two of these have Cayley index $16$, while the rest have Cayley index $8$.
Finally, in~\Cref{sec:dicyclic}, we find the Cayley index for every generalised dicyclic group that was not included in~\Cref{sec:Q}.

Much of the work that we summarise in this paper was done in~\cite{ImrichW}, but they left some gaps. Our paper fills all of these gaps, thus completing their work. Specifically, we fill the following gaps in their work. We examine the Cayley indices of the groups that do not lie in either of the infinite families; we give the Cayley indices for the four abelian groups for which they did not specify it (although they stated that these had been found by computer); we find the precise Cayley index for generalised dicyclic groups of order at most 96 (they bounded almost all of these by 4, but most in fact have Cayley index $2$); and we find the Cayley indices for all hamiltonian $2$-groups (they bounded these by 16, but almost all have Cayley index $8$). \Cref{table} summarises this work, providing the Cayley index for every finite group. 

For a number of the small individual groups, we found MRRs using Sage~\cite{sage} and its GAP package~\cite{GAP4}. The Cayley index of any of the graphs we present can be easily checked via computer, using this or other appropriate software.

Throughout this paper, $Q_8=\{\pm 1, \pm i, \pm j, \pm k\}$ is the usual representation of the quaternion group of order $8$. We use $D_{2n}$ for $n \ge 3$ to represent the dihedral group of order $2n$. Four of the exceptional groups listed in \Cref{exceptional} we denote by $H_i$ for $i \in \{1,2,3,4\}$; a precise representation of each of these groups is given in \Cref{exceptional}. 

To represent some of our MRRs, we use cartesian products. For two graphs $\Gamma_1$ and $\Gamma_2$, the \emph{cartesian product} of $\Gamma_1$ with $\Gamma_2$ is denoted by $\Gamma_1 \square \Gamma_2$. It is the graph whose vertices are the elements of $V(\Gamma_1) \times V(\Gamma_2)$, with $(u_1,v_1)$ adjacent to $(u_2,v_2)$ if and only if either $u_1=u_2$ and $v_1$ is adjacent to $v_2$ in $\Gamma_2$, or $v_1=v_2$ and $u_1$ is adjacent to $u_2$ in $\Gamma_1$.  We say that a graph $\Gamma$ on more than one vertex is \emph{prime} with respect to the cartesian product if $\Gamma\cong\Gamma_1 \square\Gamma_2$ implies that for some $i \in \{1,2\}$, $\Gamma_i \cong \Gamma$ and $\Gamma_{2-i}$ has just one vertex. It is well-known that every graph has a unique \emph{prime factorisation} as the cartesian product of prime graphs. We say that two graphs are \emph{relatively prime} with respect to the cartesian product if they have no common factors in their prime factorisations. We sometimes simply refer to the graphs as prime or relatively prime.
 
\begin{table}[!ht]
\begin{center}
\begin{tabular}{|c|c|c|}\hline
Group & Cayley index & See \\ \hline\hline
Abelian groups &&\\ \hline
$\ZZ_2^n, n \ge 5$ & 1 & \cite{Imrich2}, 1.2 of \cite{ImrichW} \\
$\ZZ_2^3, \ZZ_4 \times \ZZ_2$ & 6 & Lemma 2.7 of \cite{ImrichW}\\
$\ZZ_2^4$& 8 & \Cref{table-abelian}\\
$\ZZ_4^2$ & 4 & \Cref{table-abelian}\\
$\ZZ_4 \times \ZZ_2^2$ & 8 & \Cref{table-abelian}\\
$\ZZ_3^2$& 8 & Lemma 2.4 of \cite{ImrichW}\\
$\ZZ_3^3$& 12 & \Cref{table-abelian}\\
all other abelian groups & 2 & Theorem 1 of \cite{ImrichW} \\ \hline
Hamiltonian $2$-groups & & \\ \hline
$Q_8$ & 16 & Lemma 2.6 of \cite{ImrichW} \\
$Q_8 \times \ZZ_2$ & 16 & \Cref{sec:Q}\\
$Q_8 \times \ZZ_2^n, n \ge 2$ & 8 & \Cref{QxZ_2^n}\\ \hline
Other generalised dicyclic groups && \\ \hline
$Dic(\ZZ_6,3,x)$ & 4 & \Cref{table-dicyclic}\\
$Dic(\ZZ_8,4,x)$ & 4 & \Cref{table-dicyclic}\\
$Dic(\ZZ_4 \times \ZZ_2,(0,1),x)$ & 4 & \Cref{table-dicyclic}\\
all other generalised dicyclic groups & 2 & \Cref{sec:dicyclic}, and Theorem 2 of \cite{ImrichW}\\ \hline
Exceptional groups &&\\ \hline
$D_6$, $D_8$, $D_{10}$ & 2 & Section 2 of \cite{ImrichW}, or \Cref{table-exceptional} \\
$A_4$ & 2 & \Cref{table-exceptional} \\
$Q_8 \times \ZZ_3$, $Q_8 \times \ZZ_4$ & 2 & \Cref{table-exceptional} \\
$H_1$ of order 16 &2& \Cref{table-exceptional} \\
$H_2$ of order 16 &2&\Cref{table-exceptional} \\
$H_3$ of order 18 &2&\Cref{table-exceptional} \\
$H_4$ of order 27 &2&\Cref{table-exceptional} \\ \hline \hline
Every group not listed above & 1 & \cite{Godsil} \\ \hline

\end{tabular}\caption{Cayley indices for all finite groups}
\label{table}
\end{center}
\end{table}

\section{Exceptional groups}\label{sec:exceptional}

We begin by listing the 13 groups that do not admit a GRR but do not lie in either of the infinite families that do not admit GRRs.

\begin{thm}[see \cite{Godsil}]\label{exceptional}
The following are the only groups that are neither generalised dicyclic nor abelian of exponent greater than $2$, yet admit no GRR:
\begin{itemize}
\item $\ZZ_2^2$, $\ZZ_2^3$, $\ZZ_2^4$;
\item $D_6$, $D_8$, $D_{10}$ where these represent the dihedral group of orders $6$, $8$, and $10$ (respectively);
\item $A_4$, the alternating group of degree $4$;
\item $H_1:=\langle a,b,c: a^2=b^2=c^2=1, abc=bca=cab\rangle$;
\item $H_2:=\langle a,b: a^8=b^2=1, bab=a^5\rangle$;
\item $H_3:=\langle a,b,c: a^3=b^3=c^2=1,ab=ba,(ac)^2=(bc)^2=e\rangle$; 
\item $H_4:=\langle a,b,c: a^3=c^3=1,ac=ca,bc=cb,b^{-1}ab=ac\rangle$;
\item $Q_8 \times \ZZ_3$, $Q_8 \times \ZZ_4$, where $Q_8$ is the quaternion group of order $8$. 
\end{itemize}
\end{thm}

The groups listed in the first bullet are abelian, and their Cayley indices are given in~\Cref{sec:abelian}.

All of the remaining groups have Cayley index $2$. Their Cayley index must be at least $2$ by~\Cref{exceptional}, since they admit no GRR. This was shown explicitly in \cite[Theorem 2]{Watkins} for the dihedral groups in the second bullet. It was shown in \cite[Proposition 3.7]{Wat-alt} for $A_4$. For the groups $H_1$ and $H_3$, it was shown in \cite[Proposition 5.3 and Theorem 2]{Wat-alt}. The group $H_2$ was dealt with in~\cite[Theorem 2 or Proposition 3.1]{Now-Wat1}, and $H_4$ in  \cite[Theorem 3]{Now-Wat2}. Finally, $Q_8 \times \ZZ_3$ and $Q_8 \times \ZZ_4$ were addressed in \cite[Theorem]{Wat-Q}.

To show that the Cayley index of each is precisely $2$, we present~\Cref{table-exceptional}. For each group, we give the connection set for a Cayley graph on that group that has Cayley index $2$. The Cayley indices of these graphs can be verified by hand or by computer.

\begin{table}[!ht]
\begin{center}
\begin{tabular}{|c|c|}\hline
Group $G$ & $S$ such that $c(\Cay(G,S))=2$ \\ \hline
$D_{2n}=\langle a,b:a^2=b^n=1, aba=b^{-1}\rangle, n \in \{3,4,5\}$ & $\{a,ab\}$  \\
$A_4$ & $\{(1~2~3)^{\pm1},(1~2)(3~4)\}$\\
 $H_1=\langle a,b,c: a^2=b^2=c^2=1, abc=bca=cab\rangle$& $\{a,b,c,(ab)^{\pm1}\}$\\
$H_2=\langle a,b: a^8=b^2=1, bab=a^5\rangle$ & $\{a^{\pm1},a^{\pm2},b\}$\\
$H_3=\langle a,b,c: a^3=b^3=c^2=1,ab=ba,(ac)^2=(bc)^2=e\rangle$ & $\{a^{\pm1},c,ac,bc\}$\\
$H_4=\langle a,b,c: a^3=c^3=1,ac=ca,bc=cb,b^{-1}ab=ac\rangle$& $\{a^{\pm1},b^{\pm1},(a^{-1}b)^{\pm 1},(bab^{-1})^{\pm1}\}$\\
$Q_8\times\ZZ_3=\langle i,j,z:z^3=1,iz=zi,jz=zj\rangle$& $\{\pm i,(iz)^{\pm 1},(jz)^{\pm 1}\}$ \\
$Q_8 \times \ZZ_4=\langle  i,j,z:z^4=1,iz=zi,jz=zj\rangle$ & $\{z^{\pm 1}, \pm i, \pm j, (iz)^{\pm 1},(-kz)^{\pm 1}\}$\\
\hline

\end{tabular}\caption{MRRs for exceptional groups}
\label{table-exceptional}
\end{center}
\end{table}

The MRRs listed in the first line of this table were also mentioned in \cite{ImrichW}.

\section{Abelian groups}\label{sec:abelian}

The Cayley index of every abelian group was determined in \cite{ImrichW}. However, for a small number of these they stated only that the Cayley index had been found by Hetzel on computer, and cite a private communication. The known results on abelian groups are as follows.

\begin{thm}[Theorem 1, Lemma 2.4, Lemma 2.7 \cite{ImrichW}]
The only finite abelian groups with a Cayley index greater than $2$ are:
\begin{itemize}
\item $\ZZ_2^3$ and $\ZZ_4 \times \ZZ_2$, for which the Cayley index is $6$, with MRR $K_2 \square K_2 \square K_2$ (the cube);
\item $\ZZ_3^2$, for which the Cayley index is $8$, with MRR $K_3 \square K_3$;
\item $\ZZ_2^4$, $\ZZ_4 \times \ZZ_2^2$, and $\ZZ_4^2$; and
\item $\ZZ_3^3$.
\end{itemize}
\end{thm}

In the rest of this section, we list the Cayley index for each of the last four groups together with an MRR for each group. The Cayley indices for these graphs and the fact that these are the Cayley indices for these groups can be verified by computer.

If $A$ is an abelian group that we are presenting as being isomorphic to $\ZZ_{i_1}\times \ldots \times \ZZ_{i_k}$, then we let $\{z_1, \ldots, z_k\}$ be the canonical generating set for this group, so $|z_j|=i_j$. We present the Cayley index and an MRR for each group in~\Cref{table-abelian}.

\begin{table}[!ht]
\begin{center}
\begin{tabular}{|c|c|c|}\hline
Group & Cayley index & Connection set for an MRR \\ \hline
$\ZZ_2^4$ &8& $\{z_1,z_2,z_3,z_4,z_1z_2,z_1z_3,z_2z_4\}$  \\
$\ZZ_4 \times \ZZ_2^2$ & 8&$\{z_1^{\pm 1}, z_2,z_3,(z_1z_2)^{\pm 1},(z_1z_3)^{\pm1}\}$\\
 $\ZZ_4^2$&4& $\{z_1^{\pm1}, z_2^{\pm1},z_1^2,(z_1z_2)^{\pm1}\}$\\
$\ZZ_3^3$ & 12&$\{z_1^{\pm1}, z_2^{\pm1}, z_3^{\pm1},(z_1z_2)^{\pm1},(z_1z_3)^{\pm1},(z_2z_3)^{\pm1}\}$\\
\hline

\end{tabular}\caption{MRRs for abelian groups not given in~\cite{ImrichW}}
\label{table-abelian}
\end{center}
\end{table}

It may seem odd that $c(\ZZ_2^4)>c(\ZZ_2^3)$. However,~\Cref{multiply} does not apply here, because neither MRR for $\ZZ_2^3$ ($K_2 \square K_2 \square K_2$ and its complement, $K_4 \square K_2$), is relatively prime to $K_2$, which is the unique connected MRR for $\ZZ_2$.

\section{The groups $Q_8 \times \ZZ_2^n$}\label{sec:Q}

In this section we deal with a particular family of generalised dicyclic groups: groups of the form $Q_8 \times \ZZ_2^n$ for some nonnegative integer $n$. These groups are also known as hamiltonian $2$-groups.

We begin with three important results from \cite{ImrichW}.

\begin{lem}[Lemma 2.6, \cite{ImrichW}]\label{Q8}
The group $Q_8$ has Cayley index $16$, with $C_4 \square \overline{K_2}$ as an MRR.
\end{lem}

\begin{lem}[Proposition 2.9, \cite{ImrichW}]\label{primeMRR}
Every group other than $\ZZ_2^2$, $\ZZ_2^3$, $\ZZ_4$, $\ZZ_4\times \ZZ_2$, and $\ZZ_3^2$ admits a connected MRR that is prime with respect to the cartesian product.
\end{lem}

\begin{lem}[Lemma 2.8, \cite{ImrichW}]\label{multiply}
Let $G_1$ and $G_2$ be groups having connected MRRs that are relatively prime with respect to the cartesian product. Then $c(G_1 \times G_2) \le c(G_1)c(G_2)$.

In fact, if $\Gamma_1$ and $\Gamma_2$ are connected MRRs for $G_1$ and $G_2$ (respectively) that are relatively prime with respect to the cartesian product, then $c(\Gamma_1 \square\Gamma_2)=c(G_1)c(G_2)$ and $\Gamma_1\square\Gamma_2$ is a Cayley graph on $G_1 \times G_2$.
\end{lem}

The following observation is made in \cite{ImrichW} and is implicit in their Theorem 2(b), which states that $c(Q_8 \times \ZZ_2^n) \le 16$ for every integer $n \ge 0$. It can be deduced from \Cref{Q8,primeMRR,multiply}, using the fact that $c(\ZZ_2)=1$.

\begin{cor}\label{no-bigger}
For every group $G \notin\{\ZZ_2^2,\ZZ_2^3,\ZZ_4,\ZZ_4\times \ZZ_2,\ZZ_3^2\}$, $c(G \times \ZZ_2)\le c(G)$.
\end{cor}

The following result is key to providing a lower bound for the Cayley index of every group $Q_8 \times \ZZ_2^n$.

\begin{prop}[Classification Theorem, \cite{CompleteCCA}]\label{prop:lower-bound}
There are 8 permutations $\varphi$ of the elements of $G=Q_8 \times \ZZ_2^n$ that fix the identity, and have the property that for every $g, h \in G$, $\varphi(gh)$ is either $\varphi(g)h$, or $\varphi(g)h^{-1}$.
\end{prop}

\begin{cor}\label{8-bound}
The Cayley index of $Q_8 \times \ZZ_2^n$ is at least $8$ for every integer $n \ge 0$.
\end{cor}

\begin{proof}
Fix $n$, and let $G=Q_8 \times \ZZ_2^n$. Let $S$ be any inverse-closed subset of $G$, and let $\Gamma=\Cay(G,S)$. Let $\varphi$ be any of the $8$ permutations given in~\Cref{prop:lower-bound}. To prove this result, it will be sufficient to show that $\varphi$ is an automorphism of $\Gamma$.

We know that for any $g \in G$, $g$ is adjacent to $gs$ if and only if $s \in S$. We also know that $\varphi(gs)$ is either $\varphi(g)s$, or $\varphi(g)s^{-1}$. Since $S$ is inverse-closed, each of these is adjacent to $\varphi(g)$ if and only if $s \in S$. Thus, $\varphi$ is indeed an automorphism of $\Gamma$.
\end{proof}

To complete this section, we note that $\overline{C_4 \square\overline{K_2}}\square K_2$ is an MRR for $Q_8 \times \ZZ_2$ with Cayley index 16, verified by computer. However, for $Q_8 \times \ZZ_2^2$, the Cayley index is $8$, with MRR $\Cay(Q_8 \times \ZZ_2^2,\{\pm i ,\pm j, \pm k, \pm iz_1, \pm kz_1z_2,z_1,z_2\})$, where $z_1$ and $z_2$ are two distinct central involutions that do not lie in $Q_8$.

Thus, using~\Cref{no-bigger} and~\Cref{8-bound} we are able to conclude the following.

\begin{prop}\label{QxZ_2^n}
For every integer $n \ge 2$, the Cayley index of $Q_8 \times \ZZ_2^n$ is $8$.
\end{prop}

\section{Other Generalised Dicyclic groups}\label{sec:dicyclic}

Imrich and Watkins \cite{ImrichW} showed that generalised dicyclic groups of order greater than 96 that are not of the form $Q_8 \times \ZZ_2^n$ have Cayley index 2. Many of the ideas from their proof  in fact apply to generalised dicyclic groups of smaller orders. We reproduce these key ideas here, without their assumptions on order. We generally need to find two elements that satisfy a number of conditions. We note that the condition $a_1 \neq ya_2$ was not listed in \cite{ImrichW} but is required; for this reason we provide a full proof of~\Cref{x-fixed-or-inverted}.

\begin{defn}
Let $Dic(A,y,x)$ be a generalised dicyclic group. We say that $(a_1, a_2) \in A\times A$ is a \emph{suitable pair} of elements of $Dic(A,y,x)$ if  for every $\{i,j\}=\{1,2\}$ we have
\begin{enumerate}[(i)]
\item $a_1 \neq a_2, ya_2$
\item $a_i^2 \neq 1,y$;
\item $a_i \neq a_j^2,ya_j^2$; and
\item $a_1a_2 \neq 1, y$.
\end{enumerate}
\end{defn}

\begin{lem}\label{A-fixed}
Let $A=\langle z_1\rangle$ be a cyclic group of order $2n \ge 10$, and let $S=\{z_1,z_1^{-1}\}$. 
Then $(z_1,z_1^{-2})$ is a suitable pair for $D=Dic(A,z_1^n,x)$. 

Also, if $\Gamma=\Cay(D,S\cup\{x,x^{-1},xz_1,x^{-1}z_1,xz_1^{-2},x^{-1}z_1^{-2}\})$ and $\varphi \in \Aut(\Gamma)_1$, then $\varphi(A)=A$.
\end{lem}

\begin{proof}
We have $y=z_1^n$. We verify the conditions for $(z_1,z_1^{-2})$ to be a suitable pair. Since $n \ge 5$, (i) and (ii) are satisfied; (iii) and (iv) are equally easy to check.

It is straightforward to verify that when $n>4$, $a^n$ is the unique vertex that has 6 common neighbours with $1$. In fact, this shows that for any vertex $v$, $vz_1^n$ is uniquely determined as the vertex that has 6 common neighbours with $v$. Since the neighbours of $1$ can be partitioned into three pairs of this sort ($\{x,x^{-1}=xz_1^n\}$, $\{xz_1, xz_1^{n+1}\}$, and $\{xz_1^{-2},xz_1^{n-2}\}$) and two elements ($z_1$ and $z_1^{-1}$) whose match in this respect ($z_1^{n+1}$, and $z_1^{n-1}$ respectively) is not a neighbour of $1$, it must be the case that  $\{z_1,z_1^{-1}\}$ and $\{x,x^{-1},xz_1,xz_1^{n+1},xz_1^{-2},xz_1^{n-2}\}$ are fixed setwise by $\varphi$. Repeating this argument shows that $\varphi(c) \in A$ for every $c \in A$. Thus, $\varphi(A)=A$.
\end{proof}

\begin{lem}\label{A-fixed-2}
Let $A=\langle z_1,z_2 \rangle$ where $|z_1|=2n \ge 6$, $|z_2|=2$, and $z_1z_2=z_2z_1$, so $A \cong \ZZ_{2n}\times \ZZ_2$. Let $S=\{z_1^{\pm1},z_2\}$. Then $(z_1,z_1^{-2})$ is a suitable pair for $D=Dic(A,b,x)$.

Also, if $\Gamma=\Cay(D,S\cup\{x,x^{-1},xz_1,x^{-1}z_1,xz_1^{-2},x^{-1}z_1^{-2}\})$ and $\varphi \in \Aut(\Gamma)_1$, then $\varphi(A)=A$.
\end{lem}

\begin{proof}
Checking the conditions for $(z_1,z_1^{-2})$ to be a suitable pair is straightforward.

Since $z_2 \in S$ is central in $D$ and $x^{-1}=xz_2$, the following pairs of neighbours of $1$ are adjacent in $\Gamma$: $\{x,x^{-1}\}$; $\{xz_1,x^{-1}z_1\}$; $\{xz_1^{-2},x^{-1}z_1^{-2}\}$. However, $z_1$, $z_1^{-1}$ and $z_2$ have no neighbours in $S$. Thus, we can distinguish the neighbours of $1$ that lie in $A$ from the neighbours of $1$ that lie in $xA$. Repeating this argument shows that $\varphi(c) \in A$ for every $c \in A$. Thus $\varphi(A)=A$.
\end{proof}

\begin{lem}\label{x-fixed-or-inverted}
Let $A$ be an abelian group that has Cayley index $2$, so we can find $S$ such that $\Delta=\Cay(A,S)$ has  Cayley index $2$.
Let $D=Dic(A,y,x)$ be a generalised dicyclic group with suitable pair $(a_1, a_2)$. Let $$\Gamma=\Cay(D, S\cup\{x,x^{-1},xa_1,x^{-1}a_1,xa_2,x^{-1}a_2\})$$ and suppose that for every $\varphi \in \Aut(\Gamma)_1$, we have $\varphi(A)=A$. If $\varphi(x)\neq 1$, then $\varphi(a)=a$, and $\varphi(xa)=(xa)^{-1}$ for every $a \in A$.
\end{lem}

\begin{proof}Throughout this proof, we use $N_X(v)$ to denote the neighbours of the vertex $v$ that lie in the subset $X$ of the vertices of $\Gamma$.
First we will show that $\varphi(x) \in \{x,x^{-1}\}$.

We are assuming that $\varphi(A)=A$, and need to show that $\varphi(x)\not\in \{xa_1, x^{-1}a_1, xa_2, x^{-1}a_2\}$. Suppose that $\varphi(x) \notin \{x,x^{-1}\}$. By symmetry, without loss of generality we may assume that $\varphi(x) = xa_1$. 

Since $\varphi(A)=A$ and the induced subgraph on $A$ is $\Delta$ which has Cayley index $2$, we know that we either have $\varphi(a)=a$ for every $a\in A$, or $\varphi(a)=a^{-1}$ for every $a \in A$. (This is always the case in a Cayley graph of Cayley index $2$ on an abelian group.) 

Since $\varphi(x)=xa_1$, $\varphi(xA)=xA$, and the induced subgraph on $xA$ is isomorphic to $\Delta$, we must have either $\varphi(xa)=xaa_1$ for every $a \in A$, or $\varphi(xa)=xa^{-1}a_1$ for every $a \in A$. 

Suppose the first of these possibilities holds, so $\varphi(xa_1)=xa_1^2$, which must therefore be a neighbour of $1$ in $xA$, and hence an element of $$N_{xA}(1)=\{x,x^{-1},xa_1,x^{-1}a_1,xa_2,x^{-1}a_2\}.$$ Each of these possibilities contradicts one of the properties of being a suitable pair: any of the first four would contradict (ii); either of the last two  contradict (iii).

If on the other hand the second possibility holds, then $\varphi(xa_2)=xa_2^{-1}a_1 \in N_{xA}(1)$. Again, each possible equality contradicts one of the properties of being a suitable pair: either of the first two contradict (i); the third or fourth each contradicts (ii); and either of the last two contradict (iii). We therefore conclude that $\varphi(x) \in \{x,x^{-1}\}$, as claimed.

Next we show that $\varphi(a)=a$ for every $a \in A$.

Observe that $$N_A(x^{-1})=N_A(x)=\{1,y,a_1,ya_1,a_2,ya_2\}.$$ Thus, since $\varphi(x) \in \{x,x^{-1}\}$, we have $\varphi(N_A(x))=N_A(x)$. If $\varphi(a)=a^{-1}$ for every $a \in A$, then this implies that $a_1^{-1}\in N_A(x)$, leading to a contradiction to the definition of a suitable pair, as above. (If $a_1^{-1}$ is any of the first four elements, this contradicts (ii); if it is either of the last two, this contradicts (iv).) Thus, we must have $\varphi(a)=a$ for every $a \in A$.

Next we show that if $\varphi(x)=x$ then $\varphi=1$. 

Since the induced subgraph on $xA$ is isomorphic to $\Delta$ and has Cayley index $2$, we must either have $\varphi(xa)=xa^{-1}$ for every $a \in A$, or $\varphi(xa)=xa$ for every $a \in A$. In the latter case, $\varphi=1$ and we are done. In the former case, we must have $\varphi(N_A(xa_1^{-1}))=N_A(xa_1)$. Observe that $a_1=xa_1^{-1}x^{-1} \in N_A(xa_1^{-1}),$ so this would imply that $$a_1=\varphi(a_1)\in N_A(xa_1)=\{a_1^{-1},ya_1^{-1},1,y,a_1^{-1}a_2,ya_1^{-1}a_2\}.$$ Similar to the arguments above, each of these possibilities contradicts some property of suitable pairs. If $a_1$ were any of the first four elements of $N_A(xa_1)$ this would contradict (i); if it were either of the last two, this would contradict (iii).

Finally, we show that if $\varphi(x)=x^{-1}$ then $\varphi(xa)=(xa)^{-1}$ for every $a \in A$. 

Again since the induced subgraph on $xA$ is isomorphic to $\Delta$ and has Cayley index $2$, we must either have $\varphi(xa)=(xa)^{-1}$ for every $a \in A$, or $\varphi(xa)=x^{-1}a^{-1}$ for every $a \in A$. In the former case we are done. In the latter case, we must have $\varphi(N_A(x^{-1}a_1^{-1}))=N_A(xa_1)$. Observe that $a_1=x^{-1}a_1^{-1}x \in N_A(x^{-1}a_1^{-1})$, so this would imply that $a_1=\varphi(a_1) \in N_A(xa_1)$, yielding the same contradiction as in the previous paragraph.
\end{proof}

\begin{prop}\label{no-cyclic}
Let $A_1=\langle z_1\rangle$ be a cyclic group of order $2n \ge 6$, and $A_2=\langle z_1,z_2\rangle$ with $|z_2|=2$ and $z_1z_2=z_2z_1$.  Let $S_1=\{z_1,z_1^{-1}\}$ and $S_2=\{z_1,z_1^{-1},z_2\}$, and let $D_1=Dic(A_1,z_1^n,x)$, and $D_2=Dic(A_2,z_2,x)$. Then $$\Gamma_i=\Cay(D_i,S_i\cup\{x,x^{-1},xz_1,xz_1^{n+1},xz_1^{-2},xz_1^{n-2}\})$$ for $i \in \{1,2\}$ is connected and has Cayley index $2$ when $n \ge 5$, and $\Gamma_2$ is connected and has Cayley index $2$ when $n \ge 3$.
\end{prop}

\begin{proof}
It is easy to see that $S_1$ is the connection set for a Cayley graph on $A_1$ with Cayley index $2$. It is slightly less obvious that $S_2$ is the connection set for a Cayley graph on $A_2$ with Cayley index $2$, but becomes clear upon noting that each $a$-edge lies in a unique $4$-cycle, while each $b$-edge lies in two $4$-cycles. Fix $i \in \{1,2\}$, and if $i=1$, ensure that $n \ge 5$.

By~\Cref{A-fixed} or~\Cref{A-fixed-2}, we know that $(z_1, z_1^{-2})$ is a suitable pair for $D_i$, and that for any $\varphi \in \Aut(\Gamma_i)_1$, $\varphi(A_i)=A_i$. By~\Cref{x-fixed-or-inverted} with $S=S_i$ and this suitable pair,  we see that there are only two possibilities for $\varphi$: $\varphi=1$, or $\varphi(a)=a$ and $\varphi(xa)=(xa)^{-1}$ for every $a \in A$. Thus, $\Gamma$ has Cayley index $2$.
\end{proof}

\begin{prop}\label{x2}
Let $A$ be an abelian group of even order that contains an involution $y$, and let $D=Dic(A,y,x)$. Suppose that $D$ has a connected MRR with Cayley index $2$. Let $A'=A \times \ZZ_2$. Then $D'=Dic(A',y,x)$ has Cayley index $2$.
\end{prop}

\begin{proof}
Observe that $D' \cong D \times \ZZ_2$. The result is now immediate from~\Cref{no-bigger}.
\end{proof}

As an immediate consequence of~\Cref{no-cyclic} and~\Cref{x2}, we obtain the following.

\begin{cor}\label{cor:dicyclic}
The following generalised dicyclic groups have Cayley index $2$:
\begin{itemize}
\item $Dic(A \times \ZZ_2^k,z_1^n,x)$ where $A=\langle z_1 \rangle \cong \ZZ_{2n}$, $n \ge 5$, and $k \ge 0$; and
\item $Dic(A \times \ZZ_2^k, z_2,x)$ where $A=\langle z_1,z_2\rangle\cong \ZZ_{2n}\times \ZZ_2$, $|z_1|=2n$, $|z_2|=2$, $n \ge 3$, and $k \ge 0$.
\end{itemize}
\end{cor}

We note that if $n$ is odd, then $\ZZ_{2n}\times \ZZ_2$ has only one automorphism class of elements of order $2$, so that~\Cref{cor:dicyclic} provides two MRRs for all such groups when $n \ge 5$.

We can determine the generalised dicyclic groups that remain by considering all abelian groups of even order at most $48$. For each group, we choose one representative for each automorphism class of elements of order $2$ to be the distinguished element $y=x^2$.  By~\Cref{cor:dicyclic}, we can ignore the cyclic groups of order at least $10$, and all but two small groups that have the form $\ZZ_{2n} \times \ZZ_2$. We also eliminate any groups that are abelian or that have the form $Q_8 \times \ZZ_2^n$. Finally, if a group has the form $D \times \ZZ_2$ for some smaller generalised dicyclic group $D$ with $c(D)=2$, then~\Cref{no-bigger} gives $c(D\times \ZZ_2)=2$, so we do not have to consider these groups either. 

We conclude this section and the paper with~\Cref{table-dicyclic}, showing the Cayley index and the connection set for an MRR for each of the 15 generalised dicyclic groups that are not already covered by any of the preceding results. For three of these groups that have the form $D \times \ZZ_2$ for some smaller generalised dicyclic group $D$, we use~\Cref{no-bigger}, but only after showing that $c(D)=2$. For these, instead of explicitly giving the connection set for an MRR, we present the group as $D\times \ZZ_2$.

As in~\Cref{sec:abelian}, if $A$ is an abelian group that we are presenting as being isomorphic to $\ZZ_{i_1}\times \ldots \times \ZZ_{i_k}$, then we let $\{z_1, \ldots, z_k\}$ be the canonical generating set for this group, so $|z_j|=i_j$.

\begin{table}[!ht]
\begin{center}
\begin{tabular}{|c|c|c|}\hline
Group & Cayley index & Connection set for an MRR \\ \hline
$Dic(\ZZ_6,3,x)$&4&$\{z_1^{\pm1},x^{\pm1}\}$ \\
$Dic(\ZZ_8,4,x)$&4&$\{z_1^{\pm1},x^{\pm1}\}$\\
$Dic(\ZZ_4 \times \ZZ_2,(0,1),x)$&4&$\{z_1^{\pm1},x^{\pm1},(z_1x)^{\pm1}\}$ \\
 $Dic(\ZZ_8 \times \ZZ_2,(4,0),x)$&2& $\{z_1^{\pm1},z_2,x^{\pm1},(z_1x)^{\pm1},(z_2x)^{\pm1}\}$\\
$Dic(\ZZ_4 \times \ZZ_4,(2,0),x)$&2& $\{z_1^{\pm1},z_2^{\pm1},(z_1z_2)^{\pm1},x^{\pm1},(z_1x)^{\pm1}\}$\\
 $Dic(\ZZ_4 \times \ZZ_2^2,(0,1,0),x)$&2&$\{z_1^{\pm1},z_3,x^{\pm1},(z_1x)^{\pm1},(z_3x)^{\pm1}\}$ \\
 $Dic(\ZZ_6 \times \ZZ_3,(3,0),x)$&2&$\{z_1^{\pm1},z_2^{\pm1},x^{\pm1},(z_2x)^{\pm1},(z_1z_2x)^{\pm1}\}$ \\
 $Dic(\ZZ_8 \times \ZZ_4,(4,0),x)$&2& $\{z_1^{\pm1},z_2^{\pm1},x^{\pm1},(z_1^6z_2^{-1}x)^{\pm1},(z_1^5z_2x)^{\pm1}\}$\\
 $Dic(\ZZ_8 \times \ZZ_4,(0,2),x)$&2&$\{z_1^{\pm1},z_2^{\pm1},x^{\pm1},(z_1^5x)^{\pm1},(z_1^3z_2x)^{\pm1}\}$ \\
 $Dic(\ZZ_4^2\times \ZZ_2,(2,0,0),x)$&2&$D\cong Dic(\ZZ_4\times\ZZ_4,(2,0),x) \times \ZZ_2$ \\
 $Dic(\ZZ_4 \times \ZZ_4\times \ZZ_2,(0,0,1),x)$&2& $\{z_1^{\pm1},z_2^{\pm1},x^{\pm1},(z_2^3x)^{\pm1},(z_1^3z_2^2x)^{\pm1}\}$\\
 $Dic(\ZZ_4 \times \ZZ_2^3,(0,1,0,0),x)$&2& $D \cong Dic(\ZZ_4\times \ZZ_2^2,(0,1,0),x)\times \ZZ_2$ \\
 $Dic(\ZZ_{12} \times \ZZ_3,(6,0),x)$&2& $\{z_1^{\pm1},z_2^{\pm1},x^{\pm1},(z_1^7z_2x)^{\pm1},(z_1^3z_2x)^{\pm1}\}$\\
 $Dic(\ZZ_6 \times \ZZ_6,(3,0),x)$&2& $D \cong Dic(\ZZ_6 \times \ZZ_3,(3,0),x)\times \ZZ_2$ \\
 $Dic(\ZZ_{12} \times \ZZ_4,(6,0),x)$&2&$\{z_1^{\pm1},z_2^{\pm1},x^{\pm1},(z_1^4z_2x)^{\pm1},(z_1^9z_2^3x)^{\pm1}\}$\\
\hline

\end{tabular}\caption{MRRs for generalised dicyclic groups}
\label{table-dicyclic}
\end{center}
\end{table}

\end{document}